\documentclass[12pt]{amsart}
\usepackage[margin=1in]{geometry}
\usepackage{amsmath}
\usepackage{amsfonts}
\usepackage{amssymb}
\usepackage{graphicx}
\usepackage{mathrsfs}
\usepackage{graphicx,color}
\usepackage[mathscr]{eucal}
\usepackage{comment}
\usepackage[all]{xy}
\usepackage{hyperref}
\hypersetup{colorlinks=true, linkcolor=blue, citecolor=blue}
\usepackage{tikz}
\usepackage{tikz-cd}
\usepackage{graphicx}
\usepackage[shortlabels]{enumitem}
\usetikzlibrary{decorations.pathreplacing}

\newtheorem{theorem}{Theorem}[section]
\newtheorem{lemma}[theorem]{Lemma}

\newtheorem{corollary}[theorem]{Corollary}

\theoremstyle{definition}
\newtheorem{definition}[theorem]{Definition}
\newtheorem{example}[theorem]{Example}

\theoremstyle{remark}
\newtheorem{remark}[theorem]{Remark}

\numberwithin{equation}{section}

\newcommand{\norm}[1]{\lVert#1\rVert}

\newcommand{\ZZ}{\mathbb{Z}}

\newcommand{\RR}{\mathbb{R}}
\newcommand{\CC}{\mathbb{C}}

\newcommand{\cA}{\mathcal{A}}
\newcommand{\cD}{\mathcal{D}}

\newcommand{\cJ}{\mathcal{J}}

\newcommand{\CZ}{\mathrm{CZ}}

\newcommand{\tel}{\mathrm{tel}}
\newcommand{\red}{\mathrm{red}}
\newcommand{\co}{\mathrm{co}}
\newcommand{\Cone}{\mathrm{Cone}}

\DeclareMathOperator{\im}{im}
\DeclareMathOperator{\coker}{coker}

\title{Symplectic cohomology with support and aspherical perturbations}

\date{\today}

\author{Yuhan Sun}
\address{Imperial College London,
Department of Mathematics,
London, SW7 2AZ, U.K}
\email{yuhan.sun@imperial.ac.uk}

\begin{document}

\begin{abstract}
    For a compact subset of an aspherical symplectic manifold, we show that the symplectic cohomology with support on it remains invariant under suitable perturbations of the symplectic structure.
\end{abstract}

\maketitle

\tableofcontents

\section{Introduction}

We study some invariants of symplectic manifolds under the general name \textit{symplectic cohomology}. Roughly speaking, for a suitable symplectic manifold $(M,\omega)$ and a subset $K$ of $M$, there exists a ring $SH^*_M(K,\omega)$. The first such definition is given by Floer-Hofer \cite{FH}, where $M=\CC^n$ and $K$ is a bounded open subset. Over the past thirty years, it has been developed intensively and played a central role in the research of symplectic topology, Hamiltonian dynamics, and mirror symmetry \cite{VitSH,SeidelSH,Se,CO}. Here we are interested in how it is sensitive to the symplectic structure $\omega$.

From now on, let $K$ be a compact subset of a closed symplectic manifold $(M,\omega)$. An acceleration datum $\cD$ for $K$ is a collection of suitable Hamiltonian functions and almost complex structures on $M$. It produces a telescope $\tel^*(\cD)$, which is a chain complex whose generators are Hamiltonian orbits, and the differentials are given by counting solutions to the Floer equations. The Hamiltonian action functional gives an action filtration on $\tel^*(\cD)$. For any real numbers $a<b$, we write $\tel^*(\cD)_{[a,b)}$ as the sub-quotient complex whose generators have actions in $[a,b)$. These telescopes form a direct system in $a$ and an inverse system in $b$, as $a\to -\infty$ and $b\to +\infty$. The symplectic cohomology with support on $K$ defined by Varolgunes \cite{Var,BSV} can be computed in the following way
\[
\begin{aligned}
    SH^*_M(K,\omega)= H(\varprojlim_b\varinjlim_a\tel^*(\cD)_{[a,b)})=H(\varprojlim_b\tel^*(\cD)_{(-\infty,b)}).
\end{aligned}
\]
The resulting homology groups are independent of the choice of acceleration data $\cD$. It is an interesting question how they depend on $\omega$. In this article, we focus on the following situation.

Let $\cJ(M)$ be the space of all almost complex structures on a smooth manifold $M$. Given a symplectic form $\omega$ on $M$, we write $\cJ_\tau(M,\omega)$ as the space of almost complex structures tamed by $\omega$, that is,
\[
\cJ_\tau(M,\omega):=\{J\in \cJ(M)\mid \omega(\cdot,J\cdot)>0\}.
\]

\begin{definition}
    Let $\omega,\omega'$ be two symplectic forms on $M$. They are called adjacent if 
    \[
    \cJ_\tau(M,\omega)\cap 
    \cJ_\tau(M,\omega')\neq \emptyset.
    \]
    
\end{definition}

Since the tameness is an open condition, we have that $\cJ_\tau(M,\omega)$ is an open subspace of $\cJ(M)$. Therefore, if $\omega,\omega'$ are sufficiently close, then they are adjacent. In particular, if we have a smooth family of symplectic forms $\{\omega_t\}_{t\in[0,1]}$, then we can connect $\omega_0,\omega_1$ by a finite sequence of adjacent symplectic forms.

\begin{definition}
    Let $U$ be an open subset of $M$. Two symplectic forms $\omega,\omega'$ on $M$ are called $U$-comparable if $\omega=\omega'$ on $U$, and for any sequence of maps $u_k: (D^2,\partial D^2)\to (M,U)$, we have $\int u^*_k\omega\to +\infty$ if and only if $\int u^*_k\omega'\to +\infty$.
\end{definition}

Our main result is the following.

\begin{theorem}\label{thm:main}
    Let $K$ be a compact subset of a closed manifold $M$. Suppose that $\omega$ and $\omega'$ are two adjacent aspherical symplectic forms on $M$ which are $U$-comparable for some neighborhood $U$ of $K$. Then there are group isomorphisms
    \[
    \begin{aligned}
        SH^{*}_M(K,\omega)\cong SH^{*}_M(K,\omega').
    \end{aligned}
    \]
\end{theorem}

Recall that a symplectic form is aspherical if it vanishes on $\pi_2(M)$. Without this condition, the above statement is not true.

\begin{example}
    Let $(S^2,\omega)$ be the two sphere with an area form. Let $K$ be a disk in $S^2$ with more than half of the total area. Then $SH^{*}_{S^2}(K,\omega)\neq 0$. Set $\omega'=(1+\rho)\omega$ where $\rho$ is a bump function supported in $S^2\setminus K$. When the integral of $\rho$ is large enough, the set $K$ becomes displaceable in $(S^2,\omega')$, which implies that $SH^{*}_{S^2}(K,\omega')=0$. 
\end{example}

The $U$-comparable condition is satisfied under certain topological assumptions.

\begin{example}
    Let $\omega$ and $\omega'$ be two aspherical symplectic forms on $M$ with $\omega=\omega'$ on $U$. 
    \begin{itemize}
        \item If $U$ is incompressible in $M$, meaning that the inclusion-induced map $\pi_1(U)\to \pi_1(M)$ is injective, then $\omega$ and $\omega'$ are $U$-comparable.
        \item If $\omega$ and $\omega'$ are positively proportional on $\pi_2(M,U)$ then they are $U$-comparable.
    \end{itemize}
\end{example}

In the particular case where $M$ is a Riemann surface with positive genus, one can verify the above conditions by using surface topology. This matches the known results that symplectic cohomology with support on a positive genus surface is topological; see \cite[Theorem 1.31]{TVar} and \cite[Theorem 1.48]{DGPZ}. We also remark that the $U$-comparable condition is on some $U$, not on $K$.

\begin{remark}
    We state Theorem \ref{thm:main} for closed manifolds, while the proof works for open manifolds with both $\omega,\omega'$ being convex at infinity and $\omega-\omega'$ having compact support. For open manifolds, certain (non-)invariance results are known; for example, see \cite{GM,Gong,Ritter,BeRi}.
    
\end{remark}

\subsection{Outlook}

Now we provide more motivation for Theorem \ref{thm:main}.

\subsubsection{Heaviness}

In \cite{EP09}, Entov-Polterovich introduced the notion of heaviness, which reveals a remarkable hierarchy of rigidity among compact subsets of symplectic manifolds. It is known that symplectic cohomology with support determines the heaviness of a compact set \cite{OS, MSV}. Therefore, we expect the following.

\begin{corollary}\label{co:heavy}
    Under the setup of Theorem \ref{thm:main}, we have that $K$ is $\omega$-heavy if and only if it is $\omega'$-heavy.
\end{corollary}

\begin{remark}
    Strictly speaking, our Floer theory uses the Hamiltonian action filtration while \cite{AGV,MSV} uses a Novikov filtration. The above corollary depends on a modification of \cite{AGV} to the action filtration case, which should be much easier under our aspherical condition.
\end{remark}

The theory of heavy sets has many geometric and dynamical applications; see \cite{Entov} for a nice survey. In view of Corollary \ref{co:heavy}, these consequences are preserved under a suitable change of symplectic forms. For example, \cite[Proposition 6]{HLRS} gives a topological description of heavy sets in positive genus surfaces. Furthermore, we are interested in how the Hofer and spectral geometry of Hamiltonian groups change when we change the symplectic form. This will be pursued elsewhere.

\begin{remark}\label{re:semi}
    In \cite{Zhang}, Zhang studied how Hamiltonian spectral invariants change under the change of symplectic structures. It would be interesting to use the parametrized Floer continuation maps therein to deduce semi-continuity of symplectic cohomology with support on non-aspherical manifolds.
\end{remark}

\subsubsection{Relative Floer theory}

Let $(M,\omega)$ be a closed symplectic manifold and $D$ be a suitable symplectic divisor that supports the Poincar\'e dual of $[\omega]$. The complement of $D$ in $M$ can be made into a Liouville domain $W$. The symplectic topology and Floer theories between $W$ and $M$ provide an effective method \cite{Se} to study the homological mirror symmetry conjecture. In particular, Perutz-Sheridan \cite{PerutzSheridan} defined a relative Fukaya category $\mathcal{F}(W\subset M)$. If one varies the symplectic form outside $W$, the category has several nice invariance properties. One of our motivations is to realize similar properties for symplectic cohomology with support.

In this article, we work on the aspherical case and consider Floer theory over the integers. It should be fruitful to study symplectic cohomology with support over (several-variable) Novikov rings and establish a change of variable formula under the perturbation of symplectic structures; see \cite{M,PerutzSheridan}. Specializing to a particular symplectic form gives a map from the several-variable Novikov ring to the universal Novikov ring. This should be related to the semi-continuity phenomenon in Remark \ref{re:semi}.

\subsubsection{Sketch of proof}

In order to compare $SH^*_M(K,\omega)$ with $SH^*_M(K,\omega')$, we will use a fixed acceleration datum $\cD=(H_{s,t},J_{s,t})$ for both $\omega$ and $\omega'$. The adjacent condition implies that such a family of almost complex structures $J_{s,t}$ exists. However, the Hamiltonian vector fields for $H_{s,t}$ depend on the symplectic structure. To solve this, we assume that the functions are constant on $M\setminus U$. Then there is a canonical correspondence between the Hamiltonian vector fields for $H_{s,t}$ with respect to $\omega$ and $\omega'$. The main difficulty in using such functions is that they have many degenerate constant orbits. We will use action cut-offs to ignore these degenerate orbits, motivated by \cite{M,CO,SUV}. This gives a correspondence between generators of action-truncated telescopes.

Next, by the special choice of Hamiltonian functions, the Floer equations are the same with respect to both $\omega$ and $\omega'$. In particular, the part of a Floer solution outside $U$ becomes pseudo-holomorphic; see Figure \ref{fig:outcurve}. This provides a correspondence between moduli spaces of Floer solutions to define differentials in the telescopes.

Therefore, we have obtained two action-truncated telescopes containing a common group of generators with the same differentials, but with different action filtrations. Finally, the $U$-comparable condition is used to compare these two filtrations and to show this common group of generators captures the whole homology, after taking inverse limits. This produces the unfiltered isomorphisms in Theorem \ref{thm:main}.

\subsection*{Acknowledgment}
I thank Mark McLean, Bogdan Simeonov, and Umut Varolgunes for helpful discussions. I am supported by EPSRC grant EP/W015889/1.

\begin{figure}
    \centering
\begin{tikzpicture}[yscale=0.5]
  \draw[thick] (-5,6)--(6,6)--(6,-5)--(-5,-5)--(-5,6);
  \node at (4,-3) {$M$};

  \begin{scope}[rotate=90]
  \draw[blue, thick] (0,0.8) ellipse [x radius=2.45, y radius=3.3];

  \draw[blue, thick] (0.05,1.35) circle [radius=1.0];


  \coordinate (L) at (-0.75,-1.15);
  \coordinate (R) at (1.55,0.10);

  \draw[thick]
    (-0.58,-1.02)
    
    .. controls (1,-3) and (6,-6) .. (1.38,-0.03);

  \draw[thick]
    (-0.92,-1.28)
    
    .. controls (1,-4) and (8,-8) .. (1.72,0.23);

  \draw[thick] (L) ellipse [x radius=0.22, y radius=0.12];

  \draw[thick] (R) ellipse [x radius=0.22, y radius=0.12];

  \node[blue] at (-1.45,0.6) {$U$};
  \node[blue] at (0.05,1.35) {$K$};

  \end{scope}

\end{tikzpicture}
\caption{A Floer solution traveling outside $U$.}
    \label{fig:outcurve}
\end{figure}
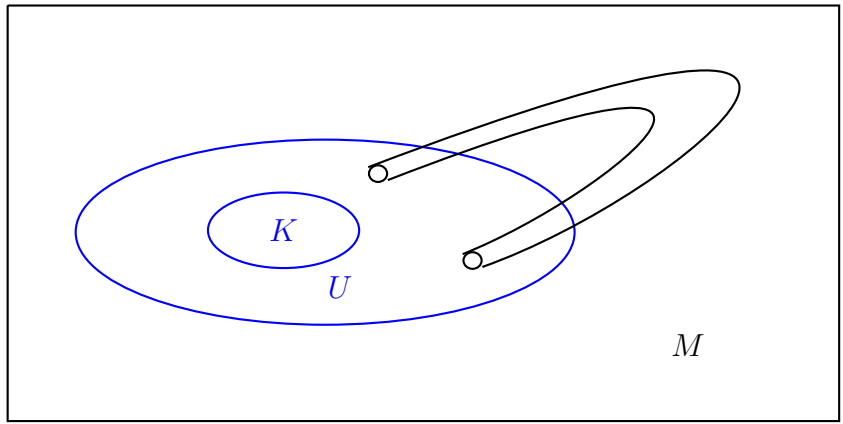

\section{Backgrounds}

Let $(M,\omega)$ be a closed manifold with an aspherical symplectic form. In this section, we review the construction of symplectic cohomology with support. 

\subsection{Hamiltonian Floer theory}

First we set up the background for Hamiltonian Floer theory. Standard references include \cite{HS,MS2,AD}. For a time-dependent Hamiltonian function $H_t:[0,1]\times M\to \RR$, its Hamiltonian vector field $X_{H_t}$ is determined by
\[
dH_t= \omega(X_{H_t},\cdot).
\]
Let $\gamma$ be a one-periodic orbit of $X_{H_t}$ that is contractible in $M$. A cap $u$ of $\gamma$ is a smooth map $u$ from the unit disk to $M$ bounding $\gamma$. The Hamiltonian action of $\gamma$ is defined as
\begin{equation}\label{eq:action}
    \cA_{H_t}(\gamma):= \int_0^1 H_t(\gamma(t))dt+ \int u^*\omega.
\end{equation}
Since $\omega$ is aspherical, the action of $\gamma$ does not depend on the cap $u$. Suppose that $\gamma$ is a non-degenerate orbit, using the cap $u$, we can define the Conley-Zehnder index $\CZ(\gamma,u)\in\ZZ$. It gives a well-defined $\ZZ/2$-grading of $\gamma$ by
\[
\deg(\gamma):=\dfrac{\dim M}{2}+\CZ(\gamma,u)\ \mod 2.
\]
If the first Chern class of $M$ vanishes on $\pi_2(M)$, then we get a well-defined $\ZZ$-grading.

Let $H_t$ be a non-degenerate Hamiltonian. We define
\begin{equation}
    CF^*(H_t):= \bigoplus_{\deg(\gamma)=*}\ZZ\langle\gamma\rangle
\end{equation}
as the free $\ZZ$-module generated by contractible one-periodic orbits of $H_t$. Given a family of tame almost complex structures $J_t$, one can consider the Floer equation
\begin{equation}\label{eq:Floer}
    \partial_sv+J_t(\partial_tv-X_{H_t})=0
\end{equation}
for smooth maps $v:\RR_s\times[0,1]_t\to M$ with asymptotics to contractible one-periodic orbits of $H_t$. A Floer solution $v$ is called regular if the linearized Floer operator of \eqref{eq:Floer} is surjective. The pair $(H_t, J_t)$ is called regular if all of its Floer solutions are regular. A regular pair $(H_t, J_t)$ gives an operator
\[
d: CF^*(H_t)\to CF^{*+1}(H_t)
\]
with $d\circ d=0$ by counting Floer solutions. We call $d$ a Floer differential.

Given two non-degenerate Hamiltonians $H_{0,t},H_{1,t}$ and a homotopy $H_{s,t}$ between them, we choose a family of almost complex structures $J_{s,t}$ such that $(H_{0,t},J_{0,t})$ and $(H_{1,t},J_{1,t})$ are regular. Then we consider a parametrized Floer equation
\begin{equation}\label{eq:paraFloer}
    \partial_sv+J_{s,t}(\partial_tv-X_{H_{s,t}})=0
\end{equation}
for $v$ with asymptotics to orbits of $H_{0,t},H_{1,t}$ respectively. The pair $(H_{s,t},J_{s,t})$ is called regular if all Floer solutions to \eqref{eq:paraFloer} are regular. A regular pair of homotopies defines a continuation operator
\[
h: CF^*(H_{0,t})\to CF^*(H_{1,t})
\]
which becomes a chain map. The homotopy $H_{s,t}$ is called monotone if $\partial_s H_{s,t}\geq 0$.

\begin{remark}
    In our sign convention, the operator $d$ does not decrease the Hamiltonian action. The same is true for $h$ if $H_{s,t}$ is monotone.
\end{remark}

\begin{theorem}
    [Theorem 5.1 in \cite{FHS}]\label{thm:trans1}
    Let $H_t$ be a non-degenerate Hamiltonian. The space of almost complex structures which make $(H_t, J_t)$ a regular pair is of second category.
\end{theorem}

\begin{theorem}
    [Theorem 11.1.6 in \cite{AD}]\label{thm:trans2}
    Given two non-degenerate Hamiltonians $H_{0,t},H_{1,t}$ and a family of almost complex structures $J_{s,t}$ such that $(H_{0,t},J_{0,t})$ and $(H_{1,t},J_{1,t})$ are regular. The space of homotopies $H_{s,t}$ connecting $H_{0,t},H_{1,t}$ which make the pair $(H_{s,t},J_{s,t})$ regular is of second category.
\end{theorem}

If $H_{1,t}>H_{0,t}$, then we can choose the pair $(H_{s,t},J_{s,t})$ to be both regular and monotone, since perturbations could be as small as we want.

\begin{remark}
    There is no problem using tame almost complex structures instead of compatible ones, as pointed out in \cite[Remark 5.3]{FHS}.
\end{remark}

\subsection{Symplectic cohomology with support}

Let $K$ be a compact subset of $M$. The symplectic cohomology with support $SH^*_M(K)$ is first defined in \cite{Var,BSV}. A reduced version is defined in \cite{Groman}, and another version is defined in \cite{CO} for Liouville cobordisms, which also works for compact subsets of aspherical manifolds. We expect that the proofs in this article can be modified to consider these two versions, too. Now we review the definitions.

\begin{definition}
    A Hamiltonian one-ray consists of the following data.
    \begin{itemize}
        \item A sequence of non-degenerate Hamiltonian functions $H_{n,t}, n=1,2,\cdots$.
        \item Homotopies $\{H_{s,t}\}_{s\in[1,\infty)}$ connecting $H_{n,t},H_{n+1,t}$ for any $n$. That is, $H_{s,t}\mid_{s=n}=H_{n,t}$.
        \item A family $\{J_{s,t}\}_{s\in[1,\infty)}$ of tame time-dependent almost complex structures.
    \end{itemize}
    A Hamiltonian one-ray is called monotone if $\partial_sH_{s,t}\geq 0$.
\end{definition}

\begin{definition}\label{def:accdatum}
    For a compact subset $K$ of $M$, an acceleration datum $\cD$ for $K$ is a monotone Hamiltonian one-ray which satisfies the following.
    \begin{itemize}
        \item The Hamiltonians $H_{n,t}$'s approximate from below the lower semi-continuous function which is zero on $K$ and positive infinity outside $K$.
        \item The two families $\{H_{s,t}\}, \{J_{s,t}\}$ are regular to define all the Floer differentials $d_n:CF^*(H_{n,t})\to CF^{*+1}(H_{n,t})$ and continuation maps $h_n:CF^*(H_{n,t})\to CF^*(H_{n+1,t})$.
    \end{itemize}
\end{definition}

For any $K$, it is known that the acceleration data exist and any two of them could be connected; see \cite[Section 3.4]{AGV}. From now on, we write $H_n$ for $H_{n,t}$ for brevity.

Given an acceleration datum $\cD$, we obtain a Floer one-ray
\[
CF(H_1)\xrightarrow{h_1} CF(H_2)\xrightarrow{h_2} \cdots
\]
connected by continuation maps $h_n$ induced by the homotopies $\{H_s\}$. We form the telescope chain complex $\tel^*(\cD)$ as follows. As a graded $\ZZ$-module, it is defined as
$$
\tel^*(\cD):=\bigoplus_{n=1}^\infty(CF^*(H_{n})\oplus CF^*(H_{n})[1]).
$$ 
The telescope differential $\delta$ is defined as follows. 
If $x_{n}\in CF^k(H_{n})$ then
$$
\delta x_{n}= d_{n}x_{n} \in CF^{k+1}(H_{n}), 
$$
and if $x'_{n}\in CF^k(H_{n})[1]$ then
\begin{align}\label{eq:delta}
	\delta x'_{n}&= (x'_{n}, -d_{n}x'_{n}, -h_{n}x'_{n})\\
	&\in CF^{k}(H_{n})\oplus CF^{k+1}(H_{n})[1]\oplus CF^k(H_{n+1}).
\end{align}

Next we take a degree-wise completion of the telescope. 
Every Floer complex $CF^*(H_n)$ is equipped with an action filtration as in \eqref{eq:action}. 
Every element in $\tel^*(\cD)$ is a finite sum of elements from $CF^*(H_n)$'s. 
We define the action of such a sum as the smallest action among its summand, and call it the $\min$-action $\mathcal{A}$ on $\tel^*(\cD)$. 
For any number $a\in [-\infty, \infty)$, we have subcomplexes $CF^*_{\geq a}(H_n)$ and $\tel^*(\cD)_{\geq a}$ containing elements with action greater than or equal to $a$. 
For $a<b$ we form the quotients
$$
CF^*_{[a,b)}(H_n):=CF^*_{\geq a}(H_n)/CF^*_{\geq b}(H_n), \quad \tel^*(\cD)_{[a,b)}:= \tel^*(\cD)_{\geq a}/\tel^*(\cD)_{\geq b}
$$
and write $CF^*_{<b}(H_n):= CF^*_{(-\infty,b]}(H_n), \tel^*(\cD)_{<b}:= \tel^*(\cD)_{(-\infty,b)}$. From the direct sum definition of the telescope, we can see that
\[
\tel^*(\cD)_{[a,b)} =\tel^*(CF^*_{[a,b)}(H_1)\to CF^*_{[a,b)}(H_2)\to\cdots).
\]
These action truncated telescopes form a direct system in $a$ and an inverse system in $b$. The degree-wise completion of the telescope is
$$
\widehat{\tel^k}(\cD):= \varprojlim_b \varinjlim_a\tel^k(\cD)_{[a,b)}=\varprojlim_b \tel^k(\cD)_{<b}
$$
as $b\to +\infty$ and $a\to-\infty$. A typical element in $\widehat{\tel^k}(\cD)$ can be written as
\begin{equation}\label{eq:element}
    x= \sum_{i=1}^\infty x_i, \quad x_i\in \tel^k(\cD), \quad \lim_{i\to \infty} \mathcal{A}(x_i)= +\infty.
\end{equation}
Another way to view the completion is as follows. For any $x\in\tel^k(\cD)$, we can define a non-Archimedean norm
\[
\norm{x}_\mathcal{A}:= \exp(-\mathcal{A}(x)), \quad \norm{0}_\mathcal{A}:= +\infty.
\]
Then $\widehat{\tel^k}(\cD)$ is the completion of $\tel^k(\cD)$ as a normed space. An infinite sum in \eqref{eq:element} is exactly a convergent sum.

Since every term in the $\delta$-differential preserves the action, it extends naturally to all convergent sums. The symplectic cohomology with support on $K$ can be computed as the homology of the completed telescope 
\begin{equation}\label{eq:defSH}
    SH^*_M(K)=H(\widehat{\tel^*}(\cD)):=\ker\delta/\im\delta,
\end{equation} 
with respect to the extended $\delta$-differential. 

\begin{remark}
    The original definition \cite{Var} of symplectic cohomology with support uses a Novikov filtration. Here we use the Hamiltonian action filtration as in \cite{BSV,SUV}. The two constructions give possibly different groups. But the pattern of proofs to obtain properties is the same. For example, the Mayer-Vietoris and displacement properties of the action filtration version are proved in \cite{SUV}.
\end{remark}

\section{Proofs}

In this section, we prove Theorem \ref{thm:main}. Recall we have two adjacent aspherical symplectic forms $\omega,\omega'$ on $M$, which are identical in a neighborhood $U$ of $K$. The closure $\Bar{U}$ is a compact codimension-zero submanifold with boundary in $M$. We assume that $\Bar{U}$ is a proper subset of $M$; otherwise $\omega=\omega'$ globally.

\subsection{Degenerate acceleration data}

In the above definition of Floer complexes, we use non-degenerate Hamiltonians. Now we introduce a certain type of degenerate Hamiltonian functions, which can be used to establish Floer complexes and are more convenient to study the symplectic perturbation problem. Our construction is motivated by \cite[Section 3.4]{SUV}.

\begin{definition}
    Let $b$ be a positive real number. A Hamiltonian $H_t$ is called $b$-admissible if 
    \begin{itemize}
        \item Any contractible one-periodic orbit of $H_t$ is either contained in $U$ or $M\setminus U$.
        \item Contractible one-periodic orbits of $H_t$ that are in $M\setminus U$ have action larger than $b$.
        \item Contractible one-periodic orbits of $H_t$ with action less than $b$ are non-degenerate.
        \item Inside the free loop space $\mathcal{LS}(M):=C^\infty(S^1,M)$ with $C^0$-norm, the locus $\mathcal{P}_{< b}(H_t)$ of one-periodic orbits of action less than $b$ and the locus $\mathcal{P}_{\geq b}(H_t)$ of one-periodic orbits of action at least equal to $b$ are isolated in the sense that their $\epsilon$-neighborhoods are disjoint for some $\epsilon>0$.
    \end{itemize}
\end{definition}

If $H_t$ is $b$-admissible, then it is $b'$-admissible for any $b'<b$.

\begin{example}\label{exam:typical}
    Let $H_t$ be a Hamiltonian such that
    \begin{itemize}
        \item It is constant on $M\setminus U$ with a value larger than $b$.
        \item It is non-degenerate on $U$.
    \end{itemize} 
    Then $H_t$ is $b$-admissible. It will be a typical Hamiltonian that we use later to construct acceleration data; see Figure \ref{fig:typ}. Every point in $M\setminus U$ is a constant degenerate orbit with Hamiltonian action greater than $b$. Using this, we will show that the action-truncated Floer complex is well-defined. To construct such functions, we start with functions that are constant on $M\setminus U$ and perturb on $U$, which is allowed by \cite[Lemma 4.2.5]{Var}.
\end{example}

\begin{figure}
    \centering

\begin{tikzpicture}[yscale=0.8]

  \def\a{2.3}   
  \def\b{1.55}  
  \def\ytop{1.55}

  \draw[thick]
    (-4,3.2) -- (-2.3,3.2)
    .. controls (-2.0,3.2) and (-2.0,2.8) .. (-2.0,2.35)
    .. controls (-2.0,2.0) and (-1.6,1.95) .. (-1.35,2.10)
    .. controls (-1.05,2.28) and (-0.85,1.95) .. (-0.60,1.85)
    .. controls (-0.25,1.72) and (0.25,1.72) .. (0.60,1.85)
    .. controls (0.85,1.95) and (1.05,2.28) .. (1.35,2.10)
    .. controls (1.6,1.95) and (2.0,2.0) .. (2.0,2.35)
    .. controls (2.0,2.8) and (2.0,3.2) .. (2.3,3.2)
    -- (4,3.2);

  \draw[dashed, thick] (-\a,3.2) -- (-\a,0);
  \draw[dashed, thick] ( \a,3.2) -- ( \a,0);

  \draw[blue, thick] (0,0) ellipse [x radius=\a, y radius=\b];

  \draw[blue, thick] (0,0) circle [radius=0.72];

  \node[blue] at (0,0) {$K$};
  \node[blue] at (1.25,-0.2) {$U$};

\end{tikzpicture}

\caption{A typical Hamiltonian}
    \label{fig:typ}
\end{figure}
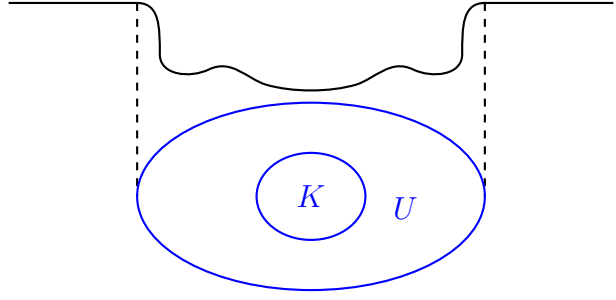

To exclude constant degenerate orbits, we need the following lemma.

\begin{lemma}\label{lem:nobreak}
    Let $b$ be a positive real number and $H_{0,t},H_{1,t}$ be two $b$-admissible Hamiltonians. Let $(H^k_{s,t}, J^k_{s,t}), k=1,2,\cdots$ be a sequence of monotone homotopies connecting $(H_{0,t},J_{0,t})$ with $(H_{1,t},J_{1,t})$ which converges to $(H^\infty_{s,t},J^\infty_{s,t})$. Let $v_k$ be a sequence of Floer solutions satisfying equation \eqref{eq:paraFloer} for $(H^k_{s,t}, J^k_{s,t})$ with fixed asymptotics whose actions less than $b$. Then, $v_k$ has a subsequence converging to a broken cylinder $v_\infty$ consisting of solutions to the Floer equation of $(H_{0,t},J_{0,t})$, then a solution to that of $(H^\infty_{s,t},J^\infty_{s,t})$ and then of $(H_{1,t},J_{1,t})$ (uniformly on all compact subsets, up to reparametrization as usual). In particular, all the one-periodic orbits that appear as an asymptotic condition in the broken solution automatically have action less than $b$.
\end{lemma}
\begin{proof}
    This is a reformulation of \cite[Lemma 3.6]{SUV}. We remark that our definition of $b$-admissible Hamiltonians is slightly more general than that in \cite{SUV}, which requires the function to be constant in some region; see Example \ref{exam:typical}. However, the proof does not need this. We also remark that the proof only uses compactness and action considerations and does not assume transversality. 
\end{proof}

Now let $H_t$ be a $b$-admissible Hamiltonian and let $CF^*_{<b}(H_t)$ be the free $\ZZ$-module generated by contractible one-periodic orbits with actions less than $b$. By assumption, all generators are in $U$ and non-degenerate. For a family $J_t$ and two orbits $\gamma_\pm\in CF^*_{<b}(H_t)$, we count the index-one Floer solutions to \eqref{eq:Floer} with asymptotics $\gamma_\pm$. Since $\gamma_\pm$ are non-degenerate, we can choose $J_t$ generic to make all Floer solutions regular. Therefore, we can define an operator
\[
d:CF^*_{<b}(H_t)\to CF^{*+1}_{<b}(H_t), \quad \gamma_-\mapsto \sum_{\gamma_+}n(\gamma_-,\gamma_+)\gamma_+
\]
where $n(\gamma_-,\gamma_+)\in \ZZ$ is the signed count of Floer solutions. Then we consider the moduli space of index-two Floer solutions with asymptotics in $CF^*_{<b}(H_t)$. By Lemma \ref{lem:nobreak}, the moduli space admits a compactification whose boundary consists of broken Floer solutions with all asymptotics in $CF^*_{<b}(H_t)$. This gives $d\circ d=0$ by the usual Floer argument.

Similarly, let $H_{0,t}<H_{1,t}$ be two $b$-admissible Hamiltonians, and let $H_{s,t}$ be a monotone homotopy connecting them. We choose a generic pair $(H_{s,t},J_{s,t})$ to count index-zero solutions to the parametrized Floer equation \eqref{eq:paraFloer} with asymptotics in $CF^*_{<b}(H_{0,t})$ and $CF^*_{<b}(H_{1,t})$ respectively. The same action argument enables us to define a continuation operator
\[
h:CF^*_{<b}(H_{0,t})\to CF^*_{<b}(H_{1,t})
\]
and Lemma \ref{lem:nobreak} shows that $h$ is a chain map.

Suppose we have three $b$-admissible Hamiltonians $H_{0,t}<H_{1,t}<H_{2,t}$, and three monotone homotopies connecting any two of them. The same argument as above shows that there exists a chain homotopy between the induced continuation operators
\[
\begin{aligned}
    & h_{02}:CF^*_{<b}(H_{0,t})\to CF^*_{<b}(H_{2,t})\\
    & h_{12}\circ h_{01}:CF^*_{<b}(H_{0,t})\to CF^*_{<b}(H_{1,t})\to CF^*_{<b}(H_{2,t}).
\end{aligned}
\]
The above definitions of the Floer differential and continuation map have a well-defined restriction to the subcomplex $CF^*_{[a,b)}(H_t)$ of $CF^*_{(-\infty,b)}(H_t)$, for any $a<b$.

Now we are ready to construct acceleration data by using possibly degenerate functions.

\begin{definition}
    Let $b$ be a positive real number. A $b$-admissible acceleration datum $\cD^b$ for $K$ is an acceleration datum as in Definition \ref{def:accdatum} such that all Hamiltonian functions $H_{n,t}$ are $b$-admissible.
\end{definition}

Given a $b$-admissible acceleration datum $\cD^b$, there is a Floer one-ray
\[
CF^*_{<b}(H_{1,t})\to CF^*_{<b}(H_{2,t})\to CF^*_{<b}(H_{3,t})\to\cdots
\]
where Floer differentials and continuation maps are defined as described above. Therefore, we can define the telescope
\begin{equation}
    \tel^*(\cD^b):=\tel^*(CF^*_{<b}(H_{1,t})\to CF^*_{<b}(H_{2,t})\to CF^*_{<b}(H_{3,t})\to\cdots).
\end{equation}

Next, we need to take the inverse limit of $\tel^*(\cD^b)$ as $b$ goes to positive infinity. 

\begin{definition}\label{def:pbdata}
    A possibly degenerate acceleration data $\cD^\mathfrak{b}_{\deg}$ for $K$ consists of
    \begin{itemize}
        \item a sequence $\mathfrak{b}$ of numbers $b_1<b_2<\cdots$ going to positive infinity, and
        \item an acceleration datum  $(H_{s,t},J_{s,t})_{s\in[1,+\infty)}$ such that for every $N=1,2,\cdots$, the sub-acceleration datum 
        \[
        \cD^{b_N}_{\deg} :=(H_{s,t},J_{s,t})_{s\in[N,+\infty)}
        \]
        is $b_N$-admissible.
    \end{itemize} 
\end{definition}

One can construct such acceleration data by using typical functions in Example \ref{exam:typical}. A particular construction will be presented later. For every $N$, there is a map
\[
\tel^*(\cD^{b_{N+1}}_{\deg})\to \tel^*(\cD^{b_{N}}_{\deg})
\]
which is the composition of the action truncation projection on Floer complexes, followed by the natural inclusion. More precisely, we have two telescopes
\[
\begin{aligned}
    & \tel^*(\cD^{b_{N}}_{\deg})=\tel^*(CF^*_{<b_N}(H_{N,t})\to CF^*_{<b_N}(H_{N+1,t})\to\cdots),\\
    & \tel^*(\cD^{b_{N+1}}_{\deg})=\tel^*(CF^*_{<b_{N+1}}(H_{N+1,t})\to CF^*_{<b_{N+1}}(H_{N+2,t})\to\cdots),
\end{aligned}
\]
on each summand $CF^*_{<b_{N+1}}(H_{N+1+k,t})$ of the telescope $\tel^*(\cD^{b_{N+1}}_{\deg})$, there is a projection
\[
CF^*_{<b_{N+1}}(H_{N+1+k,t})\to CF^*_{<b_{N}}(H_{N+1+k,t})
\]
which is action quotient in the chain level. Then we include $CF^*_{<b_{N}}(H_{N+1+k,t})$ to the telescope $\tel^*(\cD^{b_{N}}_{\deg})$. Since projection and inclusion maps are chain maps, these telescopes form an inverse system of chain complexes
\begin{equation}\label{eq:homoinverse1}
    \cdots\leftarrow \tel^*(\cD^{b_{N}}_{\deg})\leftarrow\tel^*(\cD^{b_{N+1}}_{\deg})\leftarrow\cdots.
\end{equation}
For algebraic reasons that we will see later, we use the inverse telescope
\[
\tel^*_{\leftarrow}(\cdots\leftarrow \tel^*(\cD^{b_{N}}_{\deg})\leftarrow\tel^*(\cD^{b_{N+1}}_{\deg})\leftarrow\cdots)
\]
to take the homotopy inverse limit of this system. The definition of the inverse telescope and related algebraic lemmas are reviewed in Section \ref{sec:homoinverse}.

\subsection{Comparison of Floer solutions}

Choose a particular acceleration data $\cD^\mathfrak{b}_{\deg}$ for $K$ such that 
\begin{itemize}
    \item $H_{n,t}$ satisfies Example \ref{exam:typical} for any $n=1,2,\cdots$, and
    \item $H_{s,t}=H_{1,t}+(s-1)$ on $M\setminus U$ for any $s\in[1,+\infty)$.
\end{itemize}
In particular, on $M\setminus U$, the function $H_{s,t}$ is a shift of a constant.

\begin{lemma}\label{lem:degdata}
    There exist acceleration data that both satisfy the above conditions and are regular to define the telescopes $\tel^*(\cD^{b_{N}}_{\deg})$.
\end{lemma}
\begin{proof}
    First, choose $\mathfrak{b}=(1/2,3/2,\cdots)$ and $H_{n,t}=n$ on $M\setminus U$ for each $n$. Then we extend the constant functions to $U$ to construct $H_{n,t}$ such that 
    \begin{itemize}
        \item $H_{n,t}$ are non-degenerate on $U$.
        \item $H_{n,t}< H_{n+1,t}$ for any $n$.
        \item $H_{n,t}$ converges to zero on $K$ and diverges to positive infinity outside $K$.
    \end{itemize}
    Let $\widetilde{H}_{s,t}$ be the linear interpolation from $H_{n,t}$ to $H_{n+1,t}$. In order to define the structural maps in the telescope, we only need to consider Floer solutions with asymptotics in $U$, which are all non-degenerate. Start with any family of almost complex structures $\widetilde{J}_{s,t}$, we use Theorem \ref{thm:trans1} and Theorem \ref{thm:trans2} to perturb $(\widetilde{H}_{s,t},\widetilde{J}_{s,t})$ to get regular $(H_{s,t},J_{s,t})$. The proofs therein show that the perturbations could be chosen near the asymptotics, which are all in $U$, and arbitrarily small to preserve the monotonicity of the homotopy.
\end{proof}

So far, our discussion is for a fixed symplectic form $\omega$. Now recall that we have another $\omega'$ which is adjacent to $\omega$ and has $\omega=\omega'$ on $U$. For any function $H_{n,t}$ in $\cD^\mathfrak{b}_{\deg}$, the Hamiltonian vector fields $X^\omega_{H_{n,t}}$ and $X^{\omega'}_{H_{n,t}}$ are the same. This provides a one-to-one correspondence of the one-periodic orbits. Moreover, the two (parametrized) Floer equations
\[
\partial_sv+J_{s,t}(\partial_tv-X^\omega_{H_{s,t}})=0, \quad \partial_sv+J_{s,t}(\partial_tv-X^{\omega'}_{H_{s,t}})=0
\]
become the same on $M$, which gives a one-to-one correspondence of the Floer solutions with suitable asymptotics. This is the main reason that we require functions $H_{s,t}$ being constant on $M\setminus U$. Note that by the adjacent condition, we can assume $J_{s,t}$ tame both $\omega$ and $\omega'$.

On the other hand, $\omega$ and $\omega'$ induce different Hamiltonian actions on contractible orbits. The orbits of $H_{n,t}$ outside $U$ are constant orbits, whose Hamiltonian actions do not depend on $\omega$ or $\omega'$. Therefore, if $H_{n,t}$ is $b$-admissible with respect to $\omega$, then it is also $b$-admissible with respect to $\omega'$. So, we can define both inverse systems \eqref{eq:homoinverse1} with respect to $\omega$ and $\omega'$ by using the same data $\cD^\mathfrak{b}_{\deg}$. To compare them, we introduce another telescope. For each $H_{n,t}$, we write
\[
CF^*_{<b}(H_{n,t};\omega\cap\omega'):= CF^*_{<b}(H_{n,t};\omega)\cap CF^*_{<b}(H_{n,t};\omega')
\]
as the free $\ZZ$-module generated by contractible one-periodic orbits with both $\omega$ and $\omega'$ actions less than $b$. 

\begin{lemma}
    There exist Floer differentials on $CF^*_{<b}(H_{n,t};\omega\cap\omega')$ and continuation maps $CF^*_{<b}(H_{n,t};\omega\cap\omega')\to CF^*_{<b}(H_{n+1,t};\omega\cap\omega')$ to define the telescope
    \[
    \tel^*(\cD^{b_{N}}_{\deg}; \omega\cap\omega'):=\tel^*(CF^*_{<b_N}(H_{N,t};\omega\cap\omega')\to CF^*_{<b_N}(H_{N+1,t};\omega\cap\omega')\to \cdots).
    \]
\end{lemma}
\begin{proof}
    For $\gamma_\pm\in CF^*_{<b_n}(H_{n,t};\omega\cap\omega')$, we have that 
    \[
    \cA^\omega(\gamma_\pm)<b_n, \quad \cA^{\omega'}(\gamma_\pm)<b_n.
    \]
    By our special choice of Hamiltonians and the above discussion, the two operators
    \[
    \begin{aligned}
        d^\omega: CF^*_{<b_n}(H_{n,t};\omega)\to CF^{*+1}_{<b_n}(H_{n,t};\omega)\\
        d^{\omega'}: CF^*_{<b_n}(H_{n,t};\omega')\to CF^{*+1}_{<b_n}(H_{n,t};\omega')
    \end{aligned}
    \]
    are defined by the same count 
    \[
    \gamma_-\mapsto \sum_{\gamma_+\in CF^{*+1}_{<b_n}(H_{n,t};\omega\cap\omega')}n(\gamma_-,\gamma_+)\gamma_+
    \]
    when restricted to $CF^*_{<b_n}(H_{n,t};\omega\cap\omega')$. In other words, the number $n(\gamma_-,\gamma_+)$ is the same for $\omega$ and $\omega'$. This induces a well-defined operator $d$ on $CF^*_{<b_n}(H_{n,t};\omega\cap\omega')$. To show $d$ is a differential, we note that if a Floer solution with asymptotics $\gamma_\pm$ breaks along another orbit $\beta$, then the action of $\beta$ is less than $b_n$ with respect to both $\omega$ and $\omega'$. Therefore, $d$ is a differential follows from $d^\omega=d^{\omega'}$ being a differential.

    Similarly, the restriction of the old continuation maps
    \[
    h_n: CF^*_{<b_n}(H_{n,t};\omega)\to CF^{*}_{<b_n}(H_{n+1,t};\omega)
    \]
    give well-defined continuation maps that are chain maps and satisfy the homotopy property to define the telescope.
\end{proof}

Next, we assume that $\omega$ and $\omega'$ are $U$-comparable and define a sequence $c_N$ in the following way. Set $\mathcal{P}_N:= \bigoplus_{n=N}^\infty(CF^*(H_{n,t})\oplus CF^*(H_{n,t})[1])$. For an orbit $\gamma$ of $H_{n,t}$ and a cap $u$ of $\gamma$. The action difference
\[
    \cA_{H_{n,t}}^\omega(\gamma)-\cA_{H_{n,t}}^{\omega'}(\gamma)= \int u^*\omega- \int u^*\omega'
\]
is determined by the homotopy class $[u]\in\pi_2(M,U)$, since $\omega=\omega'$ on $U$. Consider the set
\[
\mathcal{P}_N^\omega:=\{\cA^\omega(x)\mid x\in \mathcal{P}_N, \cA^\omega(x)<b_N, \cA^{\omega'}(x)\geq b_N\}\subseteq\RR.
\]
The $U$-comparable condition implies that\footnote{The infimum of the empty set is defined as $+\infty$.} 
\begin{equation}\label{eq:actionbound}
    \inf\mathcal{P}_N^\omega>-\infty.
\end{equation}
Then we make the following definition: if $\mathcal{P}_N^\omega$ is empty then we define $c_N=0$; if $\mathcal{P}_N^\omega$ is not empty, then we define
\begin{equation}\label{eqbncn}
    b_N-c_N:=\inf \mathcal{P}_N^\omega=\inf\{\cA^\omega(x)\mid x\in \mathcal{P}_N, \cA^\omega(x)<b_N, \cA^{\omega'}(x)\geq b_N\}.
\end{equation}
By \eqref{eq:actionbound}, $b_N-c_N$ is a finite number which is not greater than $b_N$. Suppose that $b_N-c_N$ stays bounded as $b_N$ diverges to positive infinity. The definition of $c_N$ gives elements $x_N$ with $\cA^\omega(x_N)$ bounded but $\cA^{\omega'}(x_N)\geq b_N$ going to positive infinity, contradicting to the $U$-comparable condition. Moreover, we can check $b_N-c_N\leq b_{N+1}-c_{N+1}$ for any $N$.
\begin{enumerate}
        \item If $\mathcal{P}_{N+1}^\omega$ is empty, then $c_{N+1}=0$ and hence $b_{N+1}-c_{N+1}=b_{N+1}\geq b_{N}-c_{N}$.
        \item If $\mathcal{P}_{N+1}^\omega$ is not empty, then there exists an element $x\in\mathcal{P}_{N+1}$
        with $\cA^\omega(x)=b_{N+1}-c_{N+1}$ and $\cA^{\omega'}(x)\geq b_{N+1}$.
        \begin{itemize}
            \item If $\cA^\omega(x)\geq b_N$, then $b_{N+1}-c_{N+1}\geq b_N\geq b_N-c_N$.
            \item If $\cA^\omega(x)< b_N$, then $\cA^\omega(x)\in\mathcal{P}_N^\omega$ and $b_N-c_N=\inf\mathcal{P}_N^\omega\leq \cA^\omega(x)$.
        \end{itemize}
\end{enumerate}

\begin{lemma}\label{lem:sequence}
There is a projection map
\[
\tel^*(\cD^{b_{N}}_{\deg}; \omega\cap\omega') \xrightarrow{p_N}
\tel^*(CF^*_{<b_N-c_N}(H_{N,t};\omega)\to CF^*_{<b_N-c_N}(H_{N+1,t};\omega)\to \cdots)
\]
which provides a commutative square
\[
\begin{tikzcd}
        & \tel^*(\cD^{b_{N+1}}_{\deg}; \omega\cap\omega') \arrow[r] \arrow[d] &  \tel^*(CF^*_{<b_{N+1}-c_{N+1}}(H_{N+1,t};\omega)\to \cdots)  \arrow[d]\\
        & \tel^*(\cD^{b_{N}}_{\deg}; \omega\cap\omega')  \arrow[r] & \tel^*(CF^*_{<b_N-c_N}(H_{N,t};\omega)\to \cdots)
\end{tikzcd}
\]
for any $N$. 
\end{lemma}
\begin{proof}
By the definition of the telescope, we have 
\[
\tel^*(\cD^{b_{N}}_{\deg}; \omega\cap\omega')= \{x\in \mathcal{P}_N\mid \cA^\omega(x), \cA^{\omega'}(x)<b_N\}.
\]
Any $x\in \tel^*(\cD^{b_{N}}_{\deg}; \omega\cap\omega')$ can be written as $x=x_0+x_+$ where $\cA^\omega(x_0)<b_N-c_N$ and $\cA^\omega(x_+)\geq b_N-c_N$. The projection $p_N$ is defined as $p_N(x):=x_0$. The definition of $b_N-c_N$ makes $p_N$ surjective. Since Floer differentials does not decrease the action, we can check that $p_N$ is a chain map and gives the commutative square. Recall that the vertical maps in the square are defined as in \eqref{eq:homoinverse1}.

\end{proof}

The above construction gives a short exact sequence
\begin{equation}
\begin{aligned}
    0 & \to \ker(p_N)\to \tel^*(\cD^{b_{N}}_{\deg}; \omega\cap\omega')\\
    & \xrightarrow{p_N}\tel^*(CF^*_{<b_N-c_N}(H_{N,t};\omega)\to CF^*_{<b_N-c_N}(H_{N+1,t};\omega)\to \cdots)\to 0.
\end{aligned}
\end{equation}
Taking the homotopy inverse limit defined in Section \ref{sec:homoinverse}, we obtain a long exact sequence
\[
\begin{aligned}
    \cdots & \to H(\tel^*_{\leftarrow}(\ker(p_N)))\to H(\tel^*_{\leftarrow}(\tel^*(\cD^{b_{N}}_{\deg}; \omega\cap\omega')))\\
    & \to H(\tel^*_{\leftarrow}(\tel^*(CF^*_{<b_N-c_N}(H_{N,t};\omega)\to CF^*_{<b_N-c_N}(H_{N+1,t};\omega)\to \cdots)))\to \cdots.
\end{aligned}
\]
See Lemma \ref{lem:inverselong}.

\begin{lemma}
    $H(\tel^*_{\leftarrow}(\ker(p_N)))=0$.
\end{lemma}
\begin{proof}
    By definition, we have
    \[
    \ker(p_{N+k})=\{x\in \mathcal{P}_{N+k}\mid \cA^\omega(x), \cA^{\omega'}(x)<b_{N+k}, \cA^\omega(x)\geq b_{N+k}-c_{N+k}\}
    \]
    for any $k=0,1,\cdots$. Since $b_N-c_N$ goes to positive infinity as $N$ goes to infinity, the map 
    \[
    \ker(p_{N+k'})\to \ker(p_{N+k})
    \]
    becomes zero when $k'-k$ is large. Therefore, the inverse system $\ker(p_N)$ satisfies the trivial Mittag-Leffler condition \cite[Definition 3.5.6]{Weibel} and hence
    \[
    H(\tel^*_{\leftarrow}(\ker(p_N)))\cong H(\varprojlim_N(\ker(p_N)))
    \]
    where the right hand side is the homology of the usual inverse limit; see Lemma \ref{lem:homoquasi}. On the other hand, the usual inverse limit $\varprojlim_N(\ker(p_N))$ is zero, by the reason that $b_N-c_N$ goes to positive infinity as $N$ goes to infinity.
\end{proof}

Redo the above construction for $\omega'$. We also get a sequence $c'_N$ such that $b_N-c'_N$ diverges to positive infinity as $b_N$ diverges to positive infinity with good properties. To summarize, we have

\begin{corollary}\label{co:homocompare}
    There exist quasi-isomorphisms
    \[
    \begin{aligned}
        \tel^*_{\leftarrow}(\tel^*(\cD^{b_{N}}_{\deg}; \omega\cap\omega'))\to \tel^*_{\leftarrow}(\tel^*(CF^*_{<b_N-c_N}(H_{N,t};\omega)\to CF^*_{<b_N-c_N}(H_{N+1,t};\omega)\to \cdots))\\
        \tel^*_{\leftarrow}(\tel^*(\cD^{b_{N}}_{\deg}; \omega\cap\omega'))\to \tel^*_{\leftarrow}(\tel^*(CF^*_{<b_N-c'_N}(H_{N,t};\omega')\to CF^*_{<b_N-c'_N}(H_{N+1,t};\omega')\to \cdots)).
    \end{aligned}
    \]
    Therefore, the groups 
    \[
    \begin{aligned}
        & H(\tel^*_{\leftarrow}(\tel^*(CF^*_{<b_N-c_N}(H_{N,t};\omega)\to CF^*_{<b_N-c_N}(H_{N+1,t};\omega)\to \cdots)))\\
        & H(\tel^*_{\leftarrow}(\tel^*(CF^*_{<b_N-c'_N}(H_{N,t};\omega')\to CF^*_{<b_N-c'_N}(H_{N+1,t};\omega')\to \cdots)))
    \end{aligned}
    \]
    are isomorphic.
\end{corollary}

It remains to compare the above homology computed by the degeneration data with the usual definition using non-degeneration data.

\subsection{Comparison of two chain models}
In this section we fix the symplectic form $\omega$ and omit it from the notation. Let $\cD^{\mathfrak{b}}_{\deg}$ be the acceleration data in Lemma \ref{lem:degdata}. There exist non-degenerate acceleration data $\cD=(G_{n,t},J'_{s,t})$ such that
\begin{itemize}
    \item All $G_{n,t}$ are non-degenerate.
    \item The homotopy $(G_{n,t},J'_{s,t})$ is regular to define the telescope $\tel^*(\cD)$.
    \item $H_{n,t}-1/4<G_{n,t}<H_{n,t}$ for all $n$.
\end{itemize}
Such data exists by perturbing the degenerate data $\cD^{\mathfrak{b}}_{\deg}$. Therefore, the original definition \eqref{eq:defSH} of symplectic cohomology with support says that
\[
SH^*_M(K)\cong H(\varprojlim_b\tel^*(\cD)_{<b}).
\]
Let $c_N$ be the sequence chosen in Lemma \ref{lem:sequence}. We can define the telescope
\[
\tel^*(CF^*_{<b_N-c_N}(G_{N,t})\to CF^*_{<b_N-c_N}(G_{N+1,t})\to\cdots).
\]
They form an inverse system with respect to $N$:
\begin{equation}\label{eq:homoinverse2}
    \begin{tikzcd}
    \cdots \arrow[d]\\
    \tel^*(CF^*_{<b_{N+1}-c_{N+1}}(G_{N+1,t})\to CF^*_{<b_{N+1}-c_{N+1}}(G_{N+2,t})\to\cdots) \arrow[d]\\
    \tel^*(CF^*_{<b_N-c_N}(G_{N,t})\to CF^*_{<b_N-c_N}(G_{N+1,t})\to\cdots)  \arrow[d]\\
    \cdots
\end{tikzcd}.
\end{equation}
Here the structural maps in the inverse system is defined as in \eqref{eq:homoinverse1}, first do action truncation then do inclusion. We write the inverse telescope of \eqref{eq:homoinverse2} as
\[
\tel^*_{\leftarrow}(\tel^*(CF^*_{<b_N-c_N}(G_{N,t}) \to \cdots)).
\]

\begin{lemma}\label{lem:compare1}
    There is an isomorphism of homology groups
    \[
    H(\varprojlim_b\tel^*(\cD)_{<b})\to H(\tel^*_{\leftarrow}(\tel^*(CF^*_{<b_N-c_N}(G_{N,t}) \to \cdots))).
    \]
\end{lemma}
\begin{proof}
    Consider another inverse system
    \begin{equation}\label{eq:homoinverse3}
    \begin{tikzcd}
    \cdots \arrow[d]\\
    \tel^*(CF^*_{<b_{N+1}-c_{N+1}}(G_{1,t})\to CF^*_{<b_{N+1}-c_{N+1}}(G_{2,t})\to\cdots) \arrow[d]\\
    \tel^*(CF^*_{<b_N-c_N}(G_{1,t})\to CF^*_{<b_N-c_N}(G_{2,t})\to\cdots)  \arrow[d]\\
    \cdots
\end{tikzcd}
\end{equation}
where the structural maps are action truncations. We write the inverse telescope of \eqref{eq:homoinverse3} as
\[
\tel^*_{\leftarrow}(\tel^*(CF^*_{<b_N-c_N}(G_{1,t}) \to \cdots)).
\]
Since action truncation maps are surjective and $\lim_N(b_N-c_N)=+\infty$, there is a quasi-isomorphism
\[
\varprojlim_b\tel^*(\cD)_{<b}= \varprojlim_N\tel^*(\cD)_{<b_N-c_N}\to \tel^*_{\leftarrow}(\tel^*(CF^*_{<b_N-c_N}(G_{1,t}) \to \cdots))
\]
by Lemma \ref{lem:homoquasi}. On the other hand, there is a map 
\begin{equation}\label{eq:finitequasi}
    \tel^*_{\leftarrow}(\tel^*(CF^*_{<b_N-c_N}(G_{N,t}) \to \cdots))\to \tel^*_{\leftarrow}(\tel^*(CF^*_{<b_N-c_N}(G_{1,t}) \to \cdots))
\end{equation}
given by inclusions
\[
g_N:\tel^*(CF^*_{<b_N-c_N}(G_{N,t}) \to \cdots)\to \tel^*(CF^*_{<b_N-c_N}(G_{1,t}) \to \cdots).
\]
Each inclusion is a quasi-isomorphism since it only misses finitely many terms and telescope is quasi-isomorphic to the usual direct limit. Therefore, we apply Corollary \ref{co:inversequasi} to show \eqref{eq:finitequasi} is a quasi-isomorphism.
\end{proof}

As in Corollary \ref{co:inversequasi}, the homotopy inverse limit of quasi-isomorphisms is automatically a quasi-isomorphism, which is not necessarily true for the usual inverse limit. This is one reason for us to use the inverse telescope.

Next we consider shifted acceleration data
\[
\cD^1:= (G_{s,t}+1,J'_{s,t}), \quad \cD^{\mathfrak{b},1}_{\deg}:= (H_{s,t}-1/4, J_{s,t}).
\]
By the definition of $H_{n,t}$ in Lemma \ref{lem:degdata}, we have $H_{n,t}-1/4$ satisfying the admissibility.

\begin{lemma}\label{lem:compare2}
    There are chain maps between four telescopes
    \[
    \begin{aligned}
        & \tel^*_{\leftarrow}(\tel^*(CF^*_{<b_N-c_N}(H_{N,t}-1/4) \to CF^*_{<b_N-c_N}(H_{N+1,t}-1/4)\to\cdots))\\
        \to\ & \tel^*_{\leftarrow}(\tel^*(CF^*_{<b_N-c_N}(G_{N,t}) \to CF^*_{<b_N-c_N}(G_{N+1,t})\to\cdots))\\
        \to\ & \tel^*_{\leftarrow}(\tel^*(CF^*_{<b_N-c_N}(H_{N,t}) \to CF^*_{<b_N-c_N}(H_{N+1,t})\to\cdots))\\
        \to\ & \tel^*_{\leftarrow}(\tel^*(CF^*_{<b_N-c_N}(G_{N,t}+1) \to CF^*_{<b_N-c_N}(G_{N+1,t}+1)\to\cdots)),
    \end{aligned}
    \]
    where the composition of the first two and the composition of the last two are unfiltered quasi-isomorphisms.
\end{lemma}
\begin{proof}
    This follows from the sandwich argument in \cite[Section 3.3.2]{Var} and \cite[Proposition 4.2.3]{VarThesis}. We briefly recall the proof here. By the monotonicity of Hamiltonians
    \[
    H_{n,t}-1/4< G_{n,t}< H_{n,t}, \quad \forall n
    \]
    we have homotopies connecting $H_{n,t}-1/4,G_{n,t}$ and connecting $G_{n,t},H_{n,t}$ respectively. These homotopies can be extended to construct chain maps between the inverse telescopes. On the other hand, the composition of the two homotopies is homotopic to the linear homotopy connecting $H_{n,t}-1/4$ and $H_{n,t}$. The linear homotopy gives a group isomorphism on the homology level, since our Hamiltonians are defined by a translation; see \cite[Proposition 4.2.3]{VarThesis}. The only difference between our setup and \cite{VarThesis} is that \cite{VarThesis} weights the Floer solutions by a Novikov variable but we don't. Therefore, the homology level map obtained in \cite[Proposition 4.2.3]{VarThesis} is the identity map weighted by a Novikov variable. What we obtained is a genuine isomorphism, which changes the filtration though. This implies that the composition of the first two chain maps is a unfiltered quasi-isomorphism. The same argument works for the composition of the last two maps. 
    
\end{proof}

A direct corollary is 
\begin{corollary}\label{co:compare}
    There is a quasi-isomorphism
    \[
    \begin{aligned}
        &\tel^*_{\leftarrow}(\tel^*(CF^*_{<b_N-c_N}(G_{N,t}) \to CF^*_{<b_N-c_N}(G_{N+1,t})\to\cdots))\\
        \to\ & \tel^*_{\leftarrow}(\tel^*(CF^*_{<b_N-c_N}(H_{N,t}) \to CF^*_{<b_N-c_N}(H_{N+1,t})\to\cdots)).
    \end{aligned}
    \]
    Therefore, we have that
    \[
    SH^*_M(K)\cong H(\tel^*_{\leftarrow}(\tel^*(CF^*_{<b_N-c_N}(H_{N,t}) \to CF^*_{<b_N-c_N}(H_{N+1,t})\to\cdots))).
    \]
\end{corollary}
\begin{proof}
    Lemma \ref{lem:compare2} says that the second map therein is a quasi-isomorphism. Then we apply Lemma \ref{lem:compare1} to finish the proof.
\end{proof}

Now we have shown that the degenerate acceleration data also compute the symplectic cohomology with support. Redo the construction in this section for $\omega'$. Theorem \ref{thm:main} follows from Corollary \ref{co:homocompare} and Corollary \ref{co:compare}.

\section{Appendix: homotopy inverse limit}\label{sec:homoinverse}

In this appendix, we collect some facts about the homotopy inverse limit. Consider an inverse system of chain complexes:
$$
\mathcal{C}: C_1^*\xleftarrow{i_{12}} C_2^* \xleftarrow{i_{23}}\cdots.
$$
We define $P^*:=\prod_l C_l^*$ as the degree-wise direct product of $C_l^*$'s. There is a chain map 
\begin{equation}\label{eq:conemap}
    f: P\to P, \quad (c_l)\mapsto (c_l-i_{l,l+1}(c_{l+1}))
\end{equation}
which gives an exact sequence
\[
0\to \ker(f)\to P\to P\to \coker(f)\to 0.
\]
The usual inverse limit of this system is defined as $\varprojlim_l C_l:= \ker(f)$, and the $\varprojlim\nolimits^1$ term is defined as $\varprojlim_l\nolimits^1 C_l:= \coker(f)$. A model for the homotopy inverse limit, the \textit{inverse telescope} complex is defined as 
$$
\tel^*_{\leftarrow}(\mathcal{C}):= \Cone(f)[-1].
$$
It always enjoys a Milnor exact sequence.

\begin{lemma}
[Lemma A.7 in \cite{BSV}]
    There is a short exact sequence
    $$
    0\to \varprojlim_l\nolimits^1 H^{j-1}(C_l^*)\to H^j(\tel^*_\leftarrow(\mathcal{C}))\to \varprojlim_l H^j(C_l^*)\to 0.
    $$
\end{lemma}

\begin{corollary}\label{co:inversequasi}
    Suppose there is a morphism between inverse systems
    \[
    \begin{tikzcd}
       & C_1^*  \arrow[d, "g_1"]  & C_2^* \arrow[l] \arrow [d, "g_2"] & \arrow[l] \cdots\\
       & D_1^*  & D_2^* \arrow[l] & \arrow[l] \cdots
    \end{tikzcd}
    \]
    such that all $g_l$ are quasi-isomorphisms. Then there is a quasi-isomorphism 
    \[
    \tel^*_\leftarrow(\mathcal{C})\to \tel^*_\leftarrow(\mathcal{D}).
    \]
\end{corollary}
\begin{proof}
    Note that $\Cone(\tel^*_\leftarrow(\mathcal{C})\to \tel^*_\leftarrow(\mathcal{D}))= \tel^*_\leftarrow(\Cone(g_l))$. The above Milnor exact sequence shows that $\Cone(\tel^*_\leftarrow(\mathcal{C})\to \tel^*_\leftarrow(\mathcal{D}))$ is acyclic.
\end{proof}

\begin{lemma}\label{lem:inverselong}
    Suppose we have three inverse systems $\mathcal{C},\mathcal{D},\mathcal{E}$ with a short exact sequence
    \[
    0\to C_l\to D_l\to E_l\to 0
    \]
    for each $l$, compatiable with structural maps. Then there is a long exact sequence
    \[
    \cdots\to H(\tel^*_{\leftarrow}(\mathcal{C}))\to H(\tel^*_{\leftarrow}(\mathcal{D}))\to H(\tel^*_{\leftarrow}(\mathcal{E}))\to \cdots.
    \]
\end{lemma}
\begin{proof}
    Taking the product, we have a short exact sequence
    \[
    0\to \prod_l C_l\to \prod_l D_l\to \prod_l E_l\to 0.
    \]
    This short exact sequence is compatible with the map $f$ in \eqref{eq:conemap}. So we get a short exact sequence of mapping cones, which induces the long exact sequence.
\end{proof}

By the relation between mapping cone and kernel, we have
\begin{lemma}\label{lem:homoquasi}
    If $\coker(f)$ is acyclic, then there exists a quasi-isomorphism 
    \[
    \varprojlim_l C_l\to \tel^*_\leftarrow(\mathcal{C}).
    \]
\end{lemma}
\begin{proof}
    This follows from the long exact sequence relating the homology groups of $\ker,\Cone$ and $\coker$; see \cite[Exercise 1.5.9]{Weibel}.
\end{proof}

A useful condition for the vanishing of $\coker(f)=\varprojlim\nolimits^1$ term is the Mittag-Leffler condition \cite[Definition 3.5.6]{Weibel}. For example, an inverse system with surjective structural maps satisfies the Mittag-Leffler condition.

\bibliography{main}

\bibliographystyle{amsplain}

\end{document}